\documentclass[12pt,bezier]{article}
\usepackage{times,tikz}
\usepackage{booktabs}
\usepackage{pifont}
\usepackage{floatrow}
\usepackage{caption}
\usepackage[colorlinks=true,anchorcolor=blue,filecolor=blue,linkcolor=blue,urlcolor=blue,citecolor=red]{hyperref}
\usepackage{mathrsfs}
\usepackage[fleqn]{amsmath}
\usepackage{amsfonts,amsthm,amssymb,mathrsfs,bbding}
\usepackage{txfonts}
\usepackage{graphics,multicol}
\usepackage{graphicx}
\usepackage{color}
\usepackage{enumerate}
\usepackage{caption}
\usepackage{cite}
\usepackage{latexsym,bm}
\usepackage{indentfirst}
\usepackage{mathtools}
\evensidemargin=\oddsidemargin\topmargin=-1.5cm

\newtheorem{thm}{Theorem}[section]
\newtheorem{prob}{Problem}
\newtheorem{lem}{Lemma}[section]
\newtheorem{cor}{Corollary}[section]

\newtheorem{remark}{Remark}

\theoremstyle{definition}

\addtocounter{section}{0}
\begin{document}
\title{The multiplicity of a Hermitian eigenvalue on graphs\footnote{This work is supported
by National Natural Science Foundation of China (No. 11371372).}}
\author{{Qian-Qian Chen,\,\, Ji-Ming Guo,\,\, Zhiwen Wang}\setcounter{footnote}{-1}\footnote{\emph{Email addresses:} qianqian\_chen@yeah.net\,(Q.-Q. Chen), jimingguo@hotmail.com\,(J.-M. Guo), walkerwzw@163.com\,(Z. Wang).}\\[2mm]
\small School of Mathematics, \\
\small  East China University of Science and Technology, Shanghai, P. R. China.}

\date{}
\maketitle

{\flushleft\large\bf Abstract}

For a graph $G$, let $\mathcal{S}(G)$ be the set consisting of Hermitian matrices whose graph is $G$.
Denoted by $m_B(G,\lambda)$ the multiplicity of an eigenvalue $\lambda$ of $B(G)\in \mathcal{S}(G)$, we show that $m_B(G,\lambda)\le 2\theta(G)+p(G)$ where $\theta(G)$ and $p(G)$ are the cyclomatic number and the number of pendent vertices of $G$ respectively, and characterize the graphs attaining the equality.
This is a generalization of a result on adjacency matrix by Wang et al.\cite{Wang1}.
Moreover, they arose an open problem in \cite{Wang1}: \textit{characterize all graphs with $m_A(G,\lambda)=2\theta(G)+p(G)-1$ for any eigenvalue $\lambda$ of its adjacency matrix.}

In this paper, we completely characterize the graphs with $m_B(G,\lambda)=2\theta(G)+p(G)-1$ for any eigenvalue $\lambda$ of an arbitrary Hermitian matrix $B(G)\in \mathcal{S}(G)$.
This result provides a stronger answer to the above problem, and encompasses some previous known works considering $\lambda=-1$ or $0$ on the problem.
\vspace{0.1cm}
\begin{flushleft}

\textbf{Keywords:} Graph; Hermitian matrix; Eigenvalue; Multiplicity.
\end{flushleft}
\textbf{AMS Classification:} 05C35, 05C50

\section{Introduction}\label{s-1}
Let $G=(V(G), E(G))$ be a finite undirected graph without loops and multiple edges, and with vertex set $V(G)$ and edge set $E(G)$, respectively. The adjacency matrix $A(G)$ of $G$ is defined as a hermitian matrix $[a_{uv}]$. For $u\sim v$, we have $a_{uv}=1$, otherwise,  $a_{uv}=0$. The nullity, denoted by $\eta(G)$, is the number of eigenvalues of $A(G)$ that are equal to $0$. The multiplicity problem of eigenvalues of graphs is a classical topic in the spectrum theory of graphs. This problem is related to the development of algebraic graph theory, which explores the relationship between graph properties and algebraic structures associated with the graph. For example, the complete graph $K_n$ has $-1$ as an eigenvalue with maximum multiplicity among all connected graphs. Additionally,  the nullity of a graph has important applications in chemistry. According to H\"{u}ckel molecular orbital theory, if the nullity of a molecular graph $G$ is positive, then the corresponding chemical compound in highly reactive , unstable, or nonexistent(see \cite{Atkins} \cite{CV}).

Regarding the adjacency matrix, much attention has been paid on exploring  the connections between various structural parameters (such as matching number, fractional matching number, chromatic number and so on) and the multiplicity of eigenvalues within graphs (see \cite{Belardo, Chang, Chang1, Chen, Gong, Guo, Ma, Wang1, Wang3, Wang4, Wang5, Zhou}  and references therein).
In 2016, a significant work by Ma et al. \cite{Ma} introduced the nullity of a general graph of cyclomatic number $\theta(G)$ in terms of the number $p(G)$ of pendant vertices, where $\theta(G)=|E(G)|-|V(G)|+\omega(G)$ and $\omega(G)$ is the number of connected components of $G$.
Recently, Wang et al.\cite{Wang1} extended this result from the nullity to the multiplicity of arbitrary eigenvalue and gave the following result:
\begin{thm}\label{the-bound-of-A}
Let $G$ be a connected graph with at least two vertices. Then $m(G, \lambda) \leq$ $2 \theta(G)+p(G)$ for any $\lambda \in \mathbb{R}$, the equality holds if and only if $G$ is a cycle $C_n$ and $\lambda=2 \cos \frac{2 k \pi}{n}$ with $k=1,2, \ldots,\left\lceil\frac{n}{2}\right\rceil-1$.
\end{thm}
In sequel, Xu et al. \cite{Xu} and Li et al. \cite{Li} further expanded Theorem \ref{the-bound-of-A} by encompassing the $A_\alpha(G)$-matrix and $A_\alpha(\mathbf{X}_G)$-matrix, respectively. Here, $\mathbf{X}_G$ refers to a mixed graph or a complex unit gain graph that has an underlying graph of $G$.
Based on Theorem \ref{the-bound-of-A}, it can be observed that if $G\ncong C_n$, then $m(G, \lambda) \leq 2 \theta(G)+p(G)-1$ for all connected graph with at least two vertices.
However, the graphs with $m(G, \lambda)=2 \theta(G)+p(G)-1$ have not yet been comprehensively characterized.
Indeed, characterizing all extremal graphs with the refined upper bound is much difficult.
Wang et al. \cite{Wang1} originally put forward the open problem:
\begin{prob}\label{prob-1}
Characterize the connected graphs with the eigenvalue $\lambda$ such that $m(G, \lambda)=2\theta(G)+p(G)-1$.
\end{prob}
It is also said that this problem is worth investigating even though $G$ is assumed to be a tree or a unicyclic graph.
The problem has been partially solved by several scholars.
Wang et al.\cite{Wang3} presented a characterization of trees with $m(T, -1)=p(T)-1$.
Subsequently, Zhou et al.\cite{Zhou} completely resolved this open problem for general graphs when $\lambda=-1$.
For the case $p(G)=0$, Chang et al. \cite{Chang1} characterized all leaf-free graphs with nullity $2\theta(G)-1$.
Chang et al. \cite{Chang} and Wang et al.\cite{Wang4} provided a complete characterization for all connected graphs with $\eta(G)=2\theta(G)+p(G)-1$, independently.

Let $G$ be a simple undirected graph with vertex set $V(G)=\{1, 2, \ldots, n\}$, and
let $B=(b_{ij})$ be a Hermitian matrix.
The (weighted) graph of $B$, denoted by $G(B)$, is defined as follows: the diagonal of $B=(B_{ij})$ is not restricted by $G$, $b_{ij}\neq 0$ if and only if the vertices $i$ and $j$ are adjacent.
Suppose that $\mathcal{S}(G)$ is the set of all  Hermitian matrices which share a common graph $G$ (see \cite{C.M., C.M.1}).
For a given $B\in \mathcal{S}(G)$, we denote by $\sigma(B(G))$ (or $\sigma(G)$ without conflict) the set consisting of eigenvalues of $B(G)$.
The multiplicity of $\lambda$ as an eigenvalue of $B$ is written as $m_B(G, \lambda)$. If $B$ is a $(0,1)$-matrix with zero diagonal entries, then $B$ is the adjacency matrix of $G$.

The spectra of adjacency matrix and (signless) Lalpacian matrix of graphs have been extensively studied. However, there are few studies for more general matrices. In the past few years,
motivated by the works of Genin  et al. \cite{Genin} and Parter \cite{Parter}, some researchers have begun to study the eigenvalues multiplicities of real symmetric matrices whose corresponding graph is a tree, see \cite{Johnson1, Johnson2, Johnson3}.

Motivated by the works mentioned above,   this paper establishes the validity of Theorem \ref{the-bound-of-A} for any Hermitian matrix $B\in \mathcal{S}(G)$, as follows:
\begin{thm}\label{thm3}
Let $G$ be a connected graph. For any eigenvalue $\lambda$ of  $B\in \mathcal{S}(G)$, $m_B(G, \lambda)\leq 2\theta(G)+p(G)$, and the equality holds if and only if $G\cong C_n$ and $m_B(G,\lambda)=2$.
\end{thm}
Naturally, Problem \ref{prob-1} can be extended to the Hermitian matrices:
\begin{prob}\label{prob-2}
Characterize the connected graphs with the eigenvalue $\lambda$ such that $m_B(G, \lambda)=2\theta(G)+p(G)-1$ for any eigenvalue $\lambda$ of  $B\in \mathcal{S}(G)$.
\end{prob}
In Section $4$, we completely solve Problem \ref{prob-2} and hence Problem $1$, as shown in Theorem \ref{thm-conclusion-result}. We also deduce previous known results from our solutions in those outstanding works.

The paper is organized as follows: Section \ref{s-2} presents some critical Lemmas that we need for investigations.
In Section \ref{s-3}, we first establish an upper bound for any Hermitian eigenvalue of a tree and characterize the trees attain the upper bound. Then we  prove that $m_B(G,\lambda)\le2\theta(G)+p(G)$ for any graph $G$ with $\lambda$ as its Hermitian eigenvalue, and the extremal graph must be a cycle. In Section \ref{s-4},  as a stronger answer to Problem \ref{prob-1}, which is Problem \ref{prob-2}, we characterize all graphs with $m_B(G,\lambda)=2\theta(G)+p(G)-1$ (see Theorem \ref{thm-conclusion-result}).

\section{Preliminaries}\label{s-2}
Given a graph $G$, a matrix whose graph is $G$ is denoted by $B(G)$, where $B\in \mathcal{S}(G)$.
Let $H$ be an induced subgraph of $G$ with vertex set $S$.
We denote the principal submatrix of $B$ resulting from retention of the rows and columns $S$ by $B(H)$. The following result follows immediately from the Interlacing theory of graphs (see \cite{Cv2}).

\begin{lem}\label{lem-2-11}
Let $G$ be a graph with a vertex $v$. Then $$m_B(G-v, \lambda)-1\leq m_B(G, \lambda)\leq m_B(G-v, \lambda)+1$$ for $\lambda\in \sigma(B(G))$ of any matrix $B\in \mathcal{S}(G)$.
\end{lem}
\begin{lem} \label{lem-2-12}
Let $G$ be a graph with an edge $e$. Then $$m_B(G, \lambda)\leq m_B(G-e, \lambda)+2$$ for $\lambda\in \sigma (B(G))$ of any matrix $B\in \mathcal{S}(G)$.
\end{lem}
\begin{proof}
Assuming that $u$ and $v$ are the endpoints of $e$,  we can express the Hermitian matrix of $G$ as:
\begin{equation*}
 B(G)=\left(\begin{array}{ccc}
 0 & b & \alpha^*\\
 \bar{b}&0&\beta^*\\
\alpha &\beta&B(G-u-v)\\
 \end{array}\right),
 \end{equation*}
 where the first and the second rows are indexed by $u$ and $v$, respectively; $\alpha^*$ and $\beta^*$ are conjugate transposes of $\alpha$ and $\beta$, respectively. Therefore,
 \begin{equation*}
 \lambda I-B(G-e)=\lambda I-B(G)+\left(\begin{array}{ccc}
 0 & b & 0\\
\bar{b} &0&0\\
0&0&0\\
 \end{array}\right),
 \end{equation*}
 which implies that
 \begin{equation*}
 rank(\lambda I-B(G-e))\leq rank(\lambda I-B(G))+2,
 \end{equation*}
 and hence,  $m_B(G, \lambda)\leq m_B(G-e, \lambda)+2$.
\end{proof}
\begin{lem}\label{lem-2-1}
Let $GuvH$ be the graph obtained from $G \cup H$ by adding an edge joining the vertex $u$ of $G$ to the vertex $v$ of $H$. For any eigenvalue $\lambda$ of any matrix $B\in \mathcal{S}(GuvH)$, if  $\lambda\in \sigma (B(G))$  and  $\lambda\notin \sigma(B(G-u))$, then $$m_B(GuvH, \lambda)=m_B(H-v,\lambda).$$
\end{lem}
\begin{proof}
The Hermitian matrix of $GuvH$ is defined as follows:
\begin{equation*}
 B(GuvH)=\left(\begin{array}{cccc}
 B(G-u) & \alpha & O & O\\
\alpha^* &b_{uu} & l&O \\
  O & \bar{l} & b_{vv}&\beta^*\\
  O&O&\beta&B(H-v)\\
 \end{array}\right),
 \end{equation*}
 where $l\neq 0$. Then
 \begin{equation*}
 B(GuvH)-\lambda I=\left(\begin{array}{cccc}
 B_1 & \alpha & O & O\\
\alpha^* &b_{uu}-\lambda & l&O \\
  O & \bar{l} & b_{vv}-\lambda&\beta^*\\
  O&O&\beta&B_2\\
 \end{array}\right),
 \end{equation*}
where $B_1=B(G-u)-\lambda I$ and $B_2=B(H-v)-\lambda I$. Since $\lambda\in \sigma(B(G))$  and  $\lambda\notin \sigma(B(G-u))$, it is easy to see $(\alpha^*, b_{uu}-\lambda)$ can be represented linearly by the row vectors of $[B_1 , \alpha]$. Therefore
\begin{equation*}
 B(GuvH)-\lambda I=\left(\begin{array}{cccc}
 B_1 & \alpha & O & O\\
\alpha^* &b_{uu}-\lambda & l&O \\
  O &  \bar{l}& b_{vv}-\lambda&\beta^*\\
  O&O&\beta&B_2\\
 \end{array}\right)\rightarrow
 \left(\begin{array}{cccc}
 B_1 & O & O & O\\
O&0 & l&O \\
  O & \bar{l} & b_{vv}-\lambda&\beta^*\\
  O&O&\beta&B_2\\
 \end{array}\right)=B'.
 \end{equation*}
 Let
 \begin{equation*}
Q=\left(\begin{array}{cccc}
I & O & O&O\\
O &1& l&-\frac{1}{\bar{l}}\beta^*\\
  O & 0 & 1&0\\
  O&O&O&I\\
 \end{array}\right),
 \end{equation*}
 we have
 \begin{equation*}
 Q^*B'Q=\left(\begin{array}{cccc}
B_1 & O & O & O\\
O &0& l&O \\
  O & l & 2\bar{l}l+b_{vv}-\lambda&O\\
  O&O&O&B_2\\
 \end{array}\right),
 \end{equation*}
 which implies that $rank(B')=n(G)+1+rank(B_2)$. Note that $rank(B(GuvH)-\lambda I)=rank(B')$,  we have $m_B(GuvH, \lambda)=m_B(H-v, \lambda)$.
\end{proof}

\begin{lem} [\cite{C.M.}]\label{lem-2-2}
 If $B\in\mathcal{S}(P_n)$, then $B$ has $n$ distinct real eigenvalues. Moreover, if $\lambda \in B(P_n)$, then $\lambda \notin B(P_{n-1})$.
\end{lem}

\section{Proof of Theorem \ref{thm3}}\label{s-3}
\vskip0.1in
\noindent4.1 \textit{An upper bound $m_B(T,\lambda)\le p-1$}
\vskip0.1in
da Fonseca \cite{C.M.} investigated the effect of the multiplicity of any eigenvalue of a graph by removing vertices,
we may conclude the results as the following lemma.

\begin{lem}\label{lem-removePath}
  Suppose $P:v_1v_2\cdots v_k$ ($k\ge 1$) is a path which does not intersect any cycle in $G$. Then for any $B\in \mathcal{S}(G)$,
  $m_B(G\backslash P, \lambda)\ge m_B(G, \lambda)-1$.
\end{lem}

 We call $P': v_1v_2\cdots v_{k}~(k\geq1)$ the pendant path of $G$ if $d_G(v_1)=1, d_G(v_2)=\cdots d_G(v_k)=2$, and another vertex, say $v_{k+1}$, adjacent to $v_k$, with degree $d_G(v_{k+1})\geq 3$. A major vertex is a vertex with a degree of at least $3$. Denoted by the vertex on the cycle of $G$ cycle-vertex. Given a graph $G$, denoted by $\mathcal{X}(G)=\{v| d(v)\geq3, v\in V(G)\}$ and $\mathcal{M}(G)=\{v| d(v)\geq3, v\in V(G) ~\text{and $v$ is not cycle-vertex}\}$.
\begin{thm}\label{thm-2}
Let $p\ge 2$ be an integer and $T$ be a tree with $p$ pendent vertices. If $\lambda$ is an eigenvalue of $B(T)\in \mathcal{S}(T)$,
then
\begin{equation}\label{equ-tree}
m_B(T, \lambda)\le p-1,
\end{equation}
and the equality holds if and only if $G$ is a path, or each path $P_i$ of $T-\mathcal{X}(T)$ with $\lambda\in \sigma(A(P_i))$ and any two major vertices of $\mathcal{X}(T)$ are not adjacent.
\end{thm}

\begin{proof}
  We first prove the upper bound $m_B(T, \lambda)\le p-1$.
  We will proceed by induction on the number of pendant vertices $p$ of $T$.

If $p(T)=2$, then $T\cong P_n$.  According to Lemma \ref{lem-2-2}, $m_B(T, \lambda)=1=p(T)-1$, as required. Next, assume for all tree with $p(T)\leq t-1~(t\geq3)$, $m_B(T, \lambda)\leq p(T)-1$.
Suppose $T$ is a tree with $t$ pendant vertices, we can choose a pendent path $P:u_1u_2\cdots u_k$ ($k\ge 1$) with
$d(u_1)=1$, $d(u_2)=\cdots =d(u_{k})=2$ (if exists).
Then $p(T\backslash P)=t-1< t$.
Thus, by Lemma \ref{lem-removePath} and the induction hypothesis, we have
\begin{equation*}
\begin{array}{lll}
m_B(T, \lambda)&\leq& m_B(T\backslash P, \lambda)+1\le (p(T\backslash P)-1)+1\\
&=&p(T\backslash P)=t-1=p-1.
\end{array}
\end{equation*}
The result follows by the principle of induction.

We next characterize the graphs with $m_B(T, \lambda)=p(T)-1$.
If $p(T)=2$, then $T$ is a path, and $m_B(T, \lambda)=1$ for any eigenvalue of $B(T)$ by Lemma \ref{lem-2-2}.
For $p(T)\ge 3$, we will prove by induction on $p(T)$.
If $p(T)=3$, then $m_B(T, \lambda)=2$ and there is only one major vertex of $T$, say $w~(d(w)\geq3)$.
Suppose the paths of $T-w$ are $P_{t_1}:u_1u_2\cdots u_{t_1}$, $P_{t_2}:v_1v_2\cdots v_{t_2}$ and $P_{t_3}:w_1w_2\cdots w_{t_3}$.
For the sufficiency part, if $\lambda\in \sigma(B(P_i))~(1\leq i\leq 3)$, by lemma \ref{lem-2-1}, $m_B(T, \lambda)=m_B(T-P_{t_1}-w, \lambda)=2$, as required.
For the necessity part,
since $m_B(T, \lambda)=2$, by Lemma \ref{lem-2-11}, there at least one path of $T-w$, say $P_{t_1}$ with $\lambda\in \sigma(B(P_{t_1}))$.
If there at least one path $B(P_{t_i})~(i=2,3)$ of $B(T-w)$, $\lambda\notin B(P_{t_i}) $, by Lemma \ref{lem-2-1}, $m_B(T, \lambda)=m_B(T-P_{t_1}-w, \lambda)\leq 1$, a contradiction.
Then the result holds for $p=3$.

Next, assume that the graphs with $m_B(T, \lambda)=p(T)-1~(3\leq p(T)\leq k-1)$ if and only if each path $P$ of $T-\mathcal{X}(T)$ with $\lambda\in \sigma(B(P))$ and any two major vertices of $\mathcal{X}(T)$ are not adjacent.
Suppose $T$ is a tree with $k$ pendent vertices, and the corresponding pendant paths are  $P^i(1\leq i\leq k)$.
Assume $\mathcal{X}(T)=\{x_1, x_2, \ldots, x_l\}~(l< k)$ is a major vertex set of $T$, and the  internal path (if exists) between $x_i$ and $x_j$ denoted by $P^{ij}~(1\leq i, j\leq l)$.
We first claim if $m_B(T, \lambda)=k-1$, then for each pendant path $P^i(1\leq i\leq k)$ of $T$, $\lambda\in \sigma(B(P^i))$.
By Lemma \ref{lem-removePath}, we have
\begin{equation*}
k-1=m_B(T, \lambda)\leq m_B(T-P^i, \lambda)+1\leq k-1,
\end{equation*}
which means $m_B(T-P^i, \lambda)=k-2=p(T-P^i)-1$.
Since $p(T-P^i)=k-1$ and $k\geq4$, by induction hypothesis, each path $P$ of $T-P^i-X(T-P^i)$ with $\lambda\in \sigma(B(P))$.
Therefore, for every pendant path $P^i(1\leq i\leq k)$ of $T$, $\lambda\in \sigma(B(P^i))$.
If there exists a major vertex, say $x_1$,
with $d(x_1)\geq 4$ and $P^1$ is adjacent to $x_1$, then $X(T-P^1)=\mathcal{X}(T)$.
By induction hypothesis, each paths $P$ of $T-P^1-\mathcal{X}(T)$  with $\lambda\in \sigma(B(P))$, as required.
Otherwise, any major vertex adjacent to pendant paths of $T$ with degree $3$. Suppose $x_1$ is a major vertex adjacent two pendant paths $P^1$ and $P^2$, and with internal path $P^{12}$ between $x_1$ and $x_2$ (see Fig.\,\ref{fig-1}).
Based on the above analysis, each path $P$ of $T-P^1-X(T-P^1)$ with $\lambda\in \sigma(B(P))$, where $X(T-P^1)=\mathcal{X}(T)/\{x_1\}$.
Thus, $\lambda\in \sigma(B(P_2\cup P^{12}/\{x_2\}))$.
Note that $\lambda\in \sigma(B(P^2))$.
By applying  Lemmas \ref{lem-2-1} and \ref{lem-2-2}, $\lambda\in \sigma(B(P^{12}/\{x_1, x_2\}))$, which also means $x_1\nsim x_2$, as required.

For the sufficiency part, suppose there are $l$ paths of $T-\mathcal{X}(T)$ and $|\mathcal{X}(T)|=t~(t\geq1)$.
Note that $\lambda\in \sigma(A(P_i))~(1\leq i\leq l)$ and any two major vertices of $\mathcal{X}(T)$ are not adjacent.
Obviously, $p(T)=l-t+1$.
By Lemma \ref{lem-2-11},
\begin{equation*}
\begin{array}{lll}
m_B(T, \lambda)&\geq&m_B(T-\mathcal{X}(T), \lambda)-t\\
&=&l-t\\
&=&p(T)-1.
\end{array}
\end{equation*}
Combining with (\ref{equ-tree}), the result follows.
\end{proof}

\vskip0.3in
\begin{figure}[htbp]
\centering
\begin{tikzpicture}[x=1.00mm, y=1.00mm, inner xsep=0pt, inner ysep=0pt, outer xsep=0pt, outer ysep=0pt,scale=0.7]
\path[line width=0mm] (82.50,51.53) rectangle +(97.40,40.86);
\definecolor{L}{rgb}{0,0,0}
\definecolor{F}{rgb}{0,0,0}
\path[line width=0.30mm, draw=L, fill=F] (93.02,88.93) circle (1.00mm);
\path[line width=0.30mm, draw=L, fill=F] (112.76,89.40) circle (1.00mm);
\path[line width=0.30mm, draw=L, fill=F] (112.76,69.42) circle (1.00mm);
\path[line width=0.30mm, draw=L, fill=F] (93.49,68.72) circle (1.00mm);
\path[line width=0.30mm, draw=L, fill=F] (122.87,79.06) circle (1.00mm);
\path[line width=0.30mm, draw=L, fill=F] (133.21,79.29) circle (1.00mm);
\path[line width=0.30mm, draw=L, fill=F] (152.72,79.06) circle (1.00mm);
\path[line width=0.30mm, draw=L, fill=F] (100.54,89.16) circle (0.50mm);
\path[line width=0.30mm, draw=L, fill=F] (103.83,89.16) circle (0.50mm);
\path[line width=0.30mm, draw=L, fill=F] (107.12,89.40) circle (0.50mm);
\path[line width=0.30mm, draw=L, fill=F] (100.31,68.95) circle (0.50mm);
\path[line width=0.30mm, draw=L, fill=F] (104.07,68.95) circle (0.50mm);
\path[line width=0.30mm, draw=L, fill=F] (107.83,69.19) circle (0.50mm);
\path[line width=0.30mm, draw=L, fill=F] (139.56,79.06) circle (0.50mm);
\path[line width=0.30mm, draw=L, fill=F] (144.02,79.06) circle (0.50mm);
\path[line width=0.30mm, draw=L, fill=F] (147.78,79.06) circle (0.50mm);
\path[line width=0.30mm, draw=L, fill=F] (85.50,88.69) circle (1.00mm);
\path[line width=0.30mm, draw=L, fill=F] (85.97,68.72) circle (1.00mm);
\path[line width=0.30mm, draw=L] (85.50,88.69) -- (93.49,88.93);
\path[line width=0.30mm, draw=L] (85.73,68.72) -- (94.20,68.72);
\path[line width=0.30mm, draw=L] (113.00,89.16) -- (122.87,79.29);
\path[line width=0.30mm, draw=L] (113.00,69.42) -- (123.10,78.35);
\path[line width=0.30mm, draw=L] (122.87,79.06) -- (133.92,79.06);
\path[line width=0.30mm, draw=L, fill=F] (163.06,88.46) circle (1.00mm);
\path[line width=0.30mm, draw=L, fill=F] (163.30,69.89) circle (1.00mm);
\path[line width=0.30mm, draw=L] (163.30,88.22) -- (152.95,78.82);
\path[line width=0.30mm, draw=L] (163.77,69.42) -- (163.30,69.89);
\path[line width=0.30mm, draw=L] (152.25,79.76) -- (163.77,69.66);
\draw(94.67,83.76) node[anchor=base west]{\fontsize{12.23}{17.07}\selectfont $P_1$};
\draw(94.90,63.31) node[anchor=base west]{\fontsize{12.23}{17.07}\selectfont $P_2$};
\draw(120.75,73.89) node[anchor=base west]{\fontsize{12.23}{17.07}\selectfont $x_1$};
\draw(147.78,73.42) node[anchor=base west]{\fontsize{12.23}{17.07}\selectfont $x_2$};
\draw(129.22,81.64) node[anchor=base west]{\fontsize{12.23}{17.07}\selectfont $P^{12}$};
\path[line width=0.30mm, draw=L, fill=F] (168.47,88.46) circle (0.50mm);
\path[line width=0.30mm, draw=L, fill=F] (172.93,88.46) circle (0.50mm);
\path[line width=0.30mm, draw=L, fill=F] (176.93,88.69) circle (0.50mm);
\path[line width=0.30mm, draw=L, fill=F] (167.76,69.89) circle (0.50mm);
\path[line width=0.30mm, draw=L, fill=F] (172.23,69.42) circle (0.50mm);
\path[line width=0.30mm, draw=L, fill=F] (177.40,69.66) circle (0.50mm);
\end{tikzpicture}%
\caption{ A tree $T$ described in proof of Theorem \ref{thm-2}}\label{fig-1}
\end{figure}
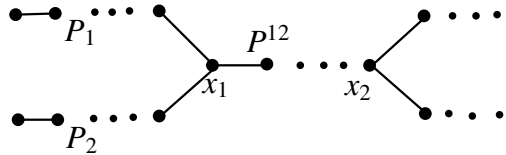
\vskip0.2in

As applications, according to Theorem \ref{thm-2}, we obtain the following results established originally in \cite{Chang} and \cite{Wang3}.
\begin{cor}[\cite{Chang}]\label{cor-1}
 Let $T$ be a tree with $p(T) \geq 3$ and $\mathcal{X}(T)$ be the set of all vertices with degree at least 3 in $T$.
 Then $\eta(T)=p(T)-1$ if and only if the following conditions are both satisfied:
\begin{enumerate}[(a)]
\vspace{-0.2cm}
   \item For any leaf $v$ of $T, d_T\left(v, \mathcal{X}(T)\right)$ is odd ;
\vspace{-0.2cm}
   \item For any $u_1, u_2 \in \mathcal{X}(T), d_T\left(u_1, u_2\right)$ is even.
\vspace{-0.2cm}
\end{enumerate}
\end{cor}
\begin{proof}
Note that $\eta(P_n)=1$ if $n$ is odd. Therefore, the condition of Theorem \ref{thm-2} is equivalent to that for any leaf $v$ of $T, d_T\left(v, \mathcal{X}(T)\right)$ is odd, and for any $u_1, u_2 \in \mathcal{X}(T), d_T\left(u_1, u_2\right)$ is even, as required.
\end{proof}

\begin{cor}[\cite{Wang3}]\label{cor-2}
Let $T$ be a tree with $p(T) \geq 2$ pendant vertices. Then $m(T,-1)\leq p(T)-1$, and the equality hold if and only if $T=P_n$ with $n \equiv 2(\bmod\, 3)$, or $T$ is a tree in which $d(v, u) \equiv 2(\bmod\, 3)$ for any pendant vertex $v$ and any major vertex $u$ of $T$.
\end{cor}
\begin{proof}
Note that $\sigma(P_n)=\{2\cos\frac{\pi j}{n+1}, j=1, \ldots, n\}$. Then, $2\cos\frac{\pi j}{n+1}=-1$ if and only if $n=3k-1~(k\geq1)$. Therefore, the condition of Theorem \ref{thm-2} is equivalent to that $d(v, u) \equiv 2(\bmod\, 3)$ for any pendant vertex $v$ and any major vertex $u$ of $T$.
\end{proof}

\vskip0.3in

\noindent4.2 \textit{Proof of Theorem \ref{thm3}}

\vskip 0.1in

An $\infty(p, q, l)$-graph is obtained from two vertex-disjoint cycles $C_p$ and $C_q$ by connecting one vertex of $C_p$ and one vertex of $C_q$ by a path of length $l-1~(l\geq2)$.
When $l=1, \infty(p, q, l)$ consists of the cycles $C_p$ and $C_q$ with one common vertex.
The $\theta(p, q, l)$-graph is union of three internally disjoint paths $P_{p+1},~P_{q+1},~P_{l+1}$ of length $p,~q,~l$, respectively with common end vertices, where $p,~q,~l\geq1$ and at most one of them is $1$ (see Fig. \ref{fig-2}).

\noindent{\bf{Proof of Theorem \ref{thm3}:}}
To prove this Theorem, we only need to show that $m_B(G, \lambda)\leq2\theta(G)+p(G)-1$ when $G\ncong C_n$. If $G$ is isomorphic to a $\theta$-graph~(see Fig. \ref{fig-2}), suppose $x$ is a vertex of $G$ with degree $3$, by Lemma \ref{lem-2-11} and Theorem \ref{thm-2}, we have
\begin{align*}
m_B(G, \lambda)&\leq m_B(G-x, \lambda)+1\\
&\leq p(G-x)-1+1\\
&=3.
\end{align*}

Next, we only consider the case $G$ is not isomorphic to a $\theta$-graph. We proceed by induction on $\theta(G)$ to prove that $m_B(G, \lambda)\leq 2\theta(G)+p(G)-1$ for any eigenvalue $\lambda\in B(G)$.
If $\theta(G)=0$, the assertion holds by Theorem \ref{thm-2}.
Assume the assertion holds for all connected graphs with $\theta(G)\leq k-1$ and let $G$ be a connected graph with $\theta(G)=k~(k\geq2)$.
If there exists a pendent cycle of $G$, say  $C: v_1v_2\cdots v_tv_1~(t\geq3)$ with $d(v_2)=\cdots =d(v_t)=2$ and $d(v_1)\geq 3$.
Since $\theta(G-v_2)=\theta(G)-1=k-1$ and $p(G-v_2)=p(G)+1$,
by Lemma \ref{lem-2-11}, we have
\begin{align*}
m_B(G, \lambda)&\leq m_B(G-v_2, \lambda)+1\\
&\leq2\theta(G-v_2)+p(G-v_2)-1+1\\
&=2(k-1)+p(G)+1-1+1\\
&=2\theta(G)+p(G)-1.
\end{align*}

Otherwise, there exists a cycle with an internal path, say $P:u_1u_2\cdots u_k$ ($k\ge 2$) , where $d(u_2)=\cdots =d(u_{k-1})=2$ (if exists), and $d(u_1), d(u_k)\geq3$. Obviously, $\theta(G-u_2)\leq \theta(G)-1=k-1$ and $p(G-u_2)\leq p(G)+1$.
{\flushleft  \bf Case 1.}  $k=2$.

Let $e=u_1u_2$. Note that $p(G-e)=p(G)$ and $G-e\ncong C_n$ for $G$ is not isomorphic to a $\theta$-graph. By Lemma \ref{lem-2-12}, we have
\begin{align*}
m_B(G, \lambda)&\leq m_B(G-e, \lambda)+2\\
&\leq2\theta(G-e)+p(G-e)-1+2\\
&\leq2(k-1)+p(G)-1+2\\
&=2\theta(G)+p(G)-1.
\end{align*}

{\flushleft  \bf Case 2.} $k\geq3$.

Note that $p(G-u_2)\leq p(G)+1$ and $G-u_2\ncong C_n$ for $G$ is not isomorphic to a $\theta$-graph. By Lemma \ref{lem-2-11}, we have
\begin{align*}
m_B(G, \lambda)&\leq m_B(G-u_2, \lambda)+1\\
&\leq2\theta(G-u_2)+p(G-u_2)-1+1\\
&\leq2(k-1)+p(G)+1\\
&=2\theta(G)+p(G)-1.
\end{align*}
The proof is complete.

A complex unit gain graph (or $\mathbb{T}$-gain graph) is a triple $\Phi=(G, \mathbb{T}, \varphi)$  consisting of a simple graph $G$, as the underlying graph of $\Phi$, the set of unit complex numbers $\mathbb{T}=\{z \in C:|z|=1\}$ and a gain function $\varphi: \vec{E} \rightarrow \mathbb{T}$ with the property that $\varphi(e_{u v})=\varphi(e_{v u})^*$.
Given a $\mathbb{T}$-gain graph $\Phi_G$, its adjacency matrix $A\left(\Phi_G\right)=\left(a_{u v}^{\varphi}\right)$ is defined as
$$
a_{u v}^{\varphi}= \begin{cases}\varphi\left(e_{u v}\right), & \text { if } u \text { and } v \text { are adjacent } \\ 0, & \text { otherwise. }\end{cases}
$$
If $u, v$ are adjacent in graph $G$, then $a_{u v}^{\varphi}=\varphi\left(e_{u v}\right)=\varphi^{-1}\left(e_{v u}\right)=\overline{\varphi\left(e_{v u}\right)}=\overline{\left(a_{v u}^{\varphi}\right)}$.
It is obvious that an undirected graph $G$ is just a complex unit gain graph $\phi$ with $\varphi(\vec{E})\subseteq\{1 \}$;
a \emph{signed graph} $(G,\sigma) $ is a complex unit gain graph $\phi$ with $\varphi(\vec{E})\subseteq\{1, -1\}$;
a \emph{mixed graph} $\tilde{G}$ is a complex unit gain graph $\phi$ with $\varphi(\vec{E})\subseteq\{1, i, -i\}$.
Let $(C_{n},\varphi)\,(n \geq 3)$ be a complex unit gain cycle. The gain of a cycle $C_n:  v_1 \cdots v_n v_1$  is denote by $\varphi(C_n)$, where
$\varphi\left(C_{n}\right)=\varphi(v_{1} v_{2}) \varphi(v_{2} v_{3}) \cdots \varphi(v_{n-1} v_{n}) \varphi(v_{n} v_{1}).$
Let $\mathfrak{D}(G)$ be the diagonal matrix whose diagonal entries are the degrees in  $G$,
and $A_\alpha\left(\Phi_G\right)=\alpha \mathfrak{D}(G)+(1-\alpha) A\left(\Phi_G\right)$, where $\alpha\in[0,1]$.
\begin{lem}[\cite{Li}]\label{Li}
Let $\Phi_{C_n}$ be a $\mathbb{T}$-gain graph with underlying graph $C_n$ and $\varphi\left(C_n\right)=\mathrm{e}^{\mathrm{i} \rho}$. Then $mult_{A_\alpha(\Phi_{C_n})}(C_n, \lambda) \leqslant 2$ for all $\lambda \in \mathbb{R}$ and $\alpha \in[0,1)$. The equality holds if and only if one of the following conditions holds:
\begin{enumerate}[(1)]
\vspace{-0.2cm}
\item $\rho=0$ and $\lambda=2 \alpha+2(1-\alpha) \cos \frac{2 j \pi}{n}, j=1,2, \ldots,\left\lceil\frac{n}{2}\right\rceil-1$;
\vspace{-0.2cm}
\item $\rho=\pi$ and $\lambda=2 \alpha+2(1-\alpha) \cos \frac{(2 j+1) \pi}{n}, j=1,2, \ldots,\left\lfloor\frac{n}{2}\right\rfloor-1$.
\vspace{-0.2cm}
\end{enumerate}
\end{lem}
According to Theorem \ref{thm3} and Lemma \ref{Li}, we may obtain a result in \cite{Li}, which is also a generalization of some results in \cite{Ma, Wang1, Xu}.

\begin{cor}[\cite{Li}]
Let $\Phi_G$ be a $\mathbb{T}$-gain graph whose underlying graph $G$ is of order at least $2$. Then
\begin{equation*}
m_{A_\alpha(\Phi_G)}(G, \lambda) \leqslant 2 \theta(G)+p(G),
\end{equation*}
for all $\lambda \in \mathbb{R}$ and $\alpha \in[0,1)$. Furthermore, the equality holds if and only if $G$ is a cycle $C_n$ and one of the following conditions holds:
\begin{enumerate}[(1)]
\vspace{-0.2cm}
\item $\varphi\left(C_n\right)=1$ and $\lambda=2 \alpha+2(1-\alpha) \cos \frac{2 j \pi}{n}, j=1,2, \ldots,\left\lceil\frac{n}{2}\right\rceil-1$;
\vspace{-0.2cm}
\item $\varphi\left(C_n\right)=-1$ and $\lambda=2 \alpha+2(1-\alpha) \cos \frac{(2 j+1) \pi}{n}, j=1,2, \ldots,\left\lfloor\frac{n}{2}\right\rfloor-1$.
\vspace{-0.2cm}
\end{enumerate}
\end{cor}

\section{Characterization of graphs with $m_B(G,\lambda)=2\theta+p-1$}\label{s-4}

We denote $\mathcal{U}$ by the set of unicyclic graphs having at most one pendant vertex. Next, we will characterize all graphs with $m_B(G, \lambda)=2\theta(G)+p(G)-1$.

For convenience, a graph $G$ is said to be $2^+$-\textit{deficient} if $m_B(G, \lambda)\leq2\theta(G)+p(G)-2$, and it is said to be $1^+$-\textit{deficient} if $m_B(G, \lambda)=2\theta(G)+p(G)-1$.

\begin{lem}\label{thm-4}
Let $p\ge 2$ be an integer and $G$ be a unicyclic graph with $p$ pendent paths. If $\lambda$ is an eigenvalue of $B(G)\in \mathcal{S}(G)$,
then $m_B(G, \lambda)= p+1$ if and only if there is only one major vertex of degree 3 in the cycle of $G$, and $G-\mathcal{M}(G)\cong C^*\cup_i P^i~(i\geq1)$, where $C^*\in \mathcal{U}$ with $m_B(C^*, \lambda)=2$ and any two major vertices of $\mathcal{M}(G)$ are not adjacent, and $P^i$ is a path with $\lambda\in \sigma(B(P^i))$ for $i\geq1$.
\end{lem}
\begin{proof}
We will prove by induction on $P(G)=p$.
If $p=2$, then $\mathcal{M}(G)\leq1$.
Suppose the two pendant paths of $G$ are $P^1:u_1u_2\cdots u_{t_1}$, $P^2:v_1v_2\cdots v_{t_2}$.
If $\mathcal{M}(G)=0$, then the major vertices adjacent to $P^1$ and $P^2$ are on the cycle, say $x_1\sim u_{t_1}$ and $x_2\sim v_{t_2}$(possibly, $x_1=x_2$), respectively.
Let $x_3$ be the cycle-vertex adjacent to $x_2$ with $d(x_3)=2$. Since $m_B(G, \lambda)=3$, we can construct an eigenvector $Y\neq \textbf{0}$ of $B(G)$ corresponding to $\lambda$ with the components $Y(v_1)=Y(x_3)=0$.
From the eigen-equation $B(T)Y=\lambda Y$, we can show that $Y=\textbf{0}$, a contradiction.
Therefore, $\mathcal{M}(G)=1$.

Let $\mathcal{M}(G)=\{u\}$, and $u\sim u_{t_1}$, $u\sim v_{t_2}$.
Note that $m_B(G, \lambda)=3$.
We can construct an eigenvector $Y'\neq \textbf{0}$ of $B(G)$ corresponding to $\lambda$ with the component $Y'(w)=Y'(z)=0$, where $w$ and $z$ are cycle-vertices of degree $2$ and $w\sim z$.
Let $Y'_1$ and $Y'_2$ be the sub-vector of $Y$ corresponding to $P^1$ and $P^2$, respectively.
Note that $Y'\neq \textbf{0}$.
From the eigen-equation $B(T)Y'=\lambda Y'$, we have $Y'_i\not=\textbf{0}$ for $i=1,2$.
Thus, $Y'_i$ is a solution of $B(P^i)Y'_i=\lambda Y'_i~(i=1,~2)$,
which implies that $\lambda\in \sigma(B(P^1))$ and  $\lambda\in \sigma(B(P^2))$.

For the necessity part, if not, then by Lemmas \ref{lem-2-1}, $m_B(G, \lambda)=m_B(G-P^1-u, \lambda)=m_B(G-P^1-P^2-u, \lambda)+m_B(P^2, \lambda)\leq 2$, where $G-P^1-P^2-u$ is a unicyclic graph with at most one pendant vertex, a contradiction.
For the sufficiency part, if $\lambda\in \sigma(B(P^i))~(i=1,2)$ and $m_B(G-P^1-P^2-u, \lambda)=2$, then by Lemmas \ref{lem-2-1}, $m_B(G, \lambda)=m_B(G-P^1-u, \lambda)=m_B(G-P^1-P^2-u, \lambda)+m_B(P^2, \lambda)= 2+1=3$, as required.
Then the result holds for $p=2$.

Next, assume the result holds for all graphs with $2\leq p\leq k-1~(k\geq3)$.
Let $G$ be an unicyclic graph with $k$ pendent paths $P^i: u^i_1u^i_2\cdots u^i_{t_i}(1\leq i\leq k)$.  For the necessity part, we first claim if $m_B(G, \lambda)=k+1$, then for each pendant path $P^i(1\leq i\leq k)$ of $G$, $\lambda\in \sigma(B(P^i))$ and there is only one major vertex in the cycle of $G$. By Lemma \ref{lem-removePath}, we have
\begin{equation*}
k+1=m_B(G, \lambda)\leq m_B(G-P^i, \lambda)+1\leq p(G-P^i)+1+1=k+1,
\end{equation*}
which means $m_B(G-P^i, \lambda)=k=p(G-P^i)+1$ for any $1\leq i\leq k$.  Since $p(G-P^i)=k-1$, by induction hypothesis, every graph $G-P^i~(1\leq i\leq  k, k\geq3)$  satisfies the condition for the equation hold. Thus, the claim holds.

If there exists a major vertex, say $x_1\in \mathcal{M}(G)$,
with $d(x_1)\geq 4$ and $P^1$ is adjacent to $x_1$, then $\mathcal{M}(G-P^1)=\mathcal{M}(G)$. By induction hypothesis,  each paths $P$ of $G-P^1-\mathcal{M}(G-P^1)$  with $\lambda\in \sigma(B(P))$, as required.  Otherwise, any major vertex of $\mathcal{M}(G)$ adjacent to pendant paths of $G$ with degree $3$. Suppose $x_1$ is a major vertex adjacent  two pendant paths $P^1$ and $P^2$ and with  internal path $P^{12}$ between $x_1$ and $x_2$ (see Fig.\,\ref{fig-1}).
According to above analysis, each path $P$ of $G-P^1-\mathcal{M}(G-P^1)$ with $\lambda\in \sigma(B(P))$, where $\mathcal{M}(G-P^1)=\mathcal{M}(G)/\{x_1\}$.
Thus, $\lambda\in \sigma(B(P_2\cup P^{12}/\{x_2\}))$.
Note that $\lambda\in \sigma(B(P^2))$.
According to Lemmas \ref{lem-2-1} and \ref{lem-2-2}, we have $\lambda\in \sigma(B(P^{12}/\{x_1, x_2\}))$, which also means $x_1\nsim x_2$, as required.

For the sufficiency part, Note that  any two major vertices of $\mathcal{M}(G)$ are not adjacent and $G-\mathcal{M}(G)\cong C^*\cup_i P^i$, where  $C^*\in\mathcal{U}$ with $m_B(C^*, \lambda)=2$.
Suppose there are $l$ paths of $G-\mathcal{M}(G)$ and $|\mathcal{M}(G)|=t$.
Obviously, $l=p+t-1$, and by Lemma \ref{lem-2-11}, we have
\begin{align*}
 m_B(G, \lambda)&\geq m_B(G-\mathcal{M}(G))-t\\
 &=2+l-t\\
 &=p+1.
\end{align*}
Combining with Theorem \ref{thm3}, we have $m_B(G, \lambda)=p+1$.
\end{proof}

\vskip0.3in
\begin{figure}[htbp]
\centering
\begin{tikzpicture}[x=1.00mm, y=1.00mm, inner xsep=0pt, inner ysep=0pt, outer xsep=0pt, outer ysep=0pt,scale=0.7]
\path[line width=0mm] (38.03,37.72) rectangle +(125.62,39.78);
\definecolor{L}{rgb}{0,0,0}
\definecolor{F}{rgb}{0,0,0}
\path[line width=0.30mm, draw=L, fill=F] (51.37,70.49) circle (1.00mm);
\path[line width=0.30mm, draw=L, fill=F] (51.37,50.51) circle (1.00mm);
\path[line width=0.30mm, draw=L, fill=F] (41.03,60.85) circle (1.00mm);
\path[line width=0.30mm, draw=L, fill=F] (60.54,60.15) circle (1.00mm);
\path[line width=0.30mm, draw=L] (51.37,70.72) -- (60.30,60.62);
\path[line width=0.30mm, draw=L] (51.61,70.72) -- (41.50,61.56);
\path[line width=0.30mm, draw=L] (60.77,60.38) -- (51.37,50.04);
\path[line width=0.30mm, draw=L, fill=F] (120.70,60.38) circle (1.00mm);
\path[line width=0.30mm, draw=L, fill=F] (130.33,70.49) circle (1.00mm);
\path[line width=0.30mm, draw=L, fill=F] (130.57,60.15) circle (1.00mm);
\path[line width=0.30mm, draw=L, fill=F] (130.80,49.80) circle (1.00mm);
\path[line width=0.30mm, draw=L, fill=F] (151.25,70.49) circle (1.00mm);
\path[line width=0.30mm, draw=L, fill=F] (151.02,60.62) circle (1.00mm);
\path[line width=0.30mm, draw=L, fill=F] (150.31,50.04) circle (1.00mm);
\path[line width=0.30mm, draw=L, fill=F] (160.65,60.38) circle (1.00mm);
\path[line width=0.30mm, draw=L] (120.46,60.38) -- (130.57,70.49);
\path[line width=0.30mm, draw=L] (120.93,60.15) -- (131.04,60.15);
\path[line width=0.30mm, draw=L] (120.93,60.38) -- (130.80,49.33);
\path[line width=0.30mm, draw=L] (151.49,70.72) -- (160.65,60.38);
\path[line width=0.30mm, draw=L] (160.89,60.85) -- (150.55,49.80);
\path[line width=0.30mm, draw=L] (151.02,60.62) -- (161.12,60.62);
\path[line width=0.45mm, draw=L, dash pattern=on 0.45mm off 0.50mm] (40.79,60.85) -- (52.08,49.80);
\path[line width=0.45mm, draw=L, dash pattern=on 0.45mm off 0.50mm] (130.10,70.25) -- (151.25,70.25);
\path[line width=0.45mm, draw=L, dash pattern=on 0.45mm off 0.50mm] (130.57,60.15) -- (151.49,60.15);
\path[line width=0.45mm, draw=L, dash pattern=on 0.45mm off 0.50mm] (130.80,49.80) -- (150.55,50.04);
\draw(45.02,59.21) node[anchor=base west]{\fontsize{11.38}{13.66}\selectfont $C_p$};
\draw(60.86,40.58) node[anchor=base west]{\fontsize{11.38}{13.66}\selectfont $\infty$-graph};
\draw(132.68,72.37) node[anchor=base west]{\fontsize{11.38}{13.66}\selectfont $P_{p+1}$};
\draw(132.92,62.50) node[anchor=base west]{\fontsize{11.38}{13.66}\selectfont $P_{q+1}$};
\draw(132.92,51.92) node[anchor=base west]{\fontsize{11.38}{13.66}\selectfont $P_{l+1}$};
\draw(133.86,40.87) node[anchor=base west]{\fontsize{11.38}{13.66}\selectfont $\theta$-graph};
\path[line width=0.30mm, draw=L, fill=F] (89.94,69.58) circle (1.00mm);
\path[line width=0.30mm, draw=L, fill=F] (89.94,49.60) circle (1.00mm);
\path[line width=0.30mm, draw=L, fill=F] (79.60,59.95) circle (1.00mm);
\path[line width=0.30mm, draw=L, fill=F] (99.10,59.24) circle (1.00mm);
\path[line width=0.30mm, draw=L] (89.94,69.82) -- (98.87,59.71);
\path[line width=0.30mm, draw=L] (90.17,69.82) -- (80.07,60.65);
\path[line width=0.30mm, draw=L] (99.34,59.48) -- (89.94,49.13);
\path[line width=0.45mm, draw=L, dash pattern=on 0.45mm off 0.50mm] (79.36,59.95) -- (90.64,48.90);
\draw(83.59,58.30) node[anchor=base west]{\fontsize{11.38}{13.66}\selectfont $C_q$};
\path[line width=0.30mm, draw=L, fill=F] (66.08,59.93) circle (1.00mm);
\path[line width=0.30mm, draw=L, fill=F] (74.95,60.11) circle (1.00mm);
\path[line width=0.30mm, draw=L] (60.83,60.11) -- (66.80,59.93);
\path[line width=0.30mm, draw=L] (75.13,60.29) -- (80.20,59.93);
\path[line width=0.30mm, draw=L, dash pattern=on 0.30mm off 0.50mm] (66.08,59.93) -- (75.13,59.93);
\draw(65.72,62.65) node[anchor=base west]{\fontsize{12.23}{17.07}\selectfont $l-1$};
\end{tikzpicture}%
\caption{$\infty(p, q, l)$-graph and $\theta(p, q, l)$-graph}\label{fig-2}
\end{figure}
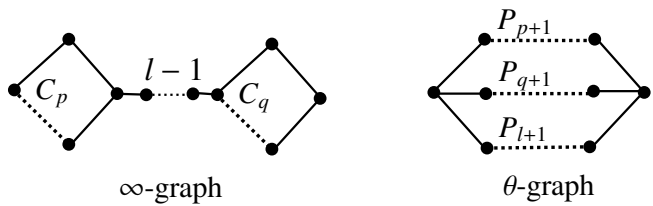
\vskip0.3in

\begin{lem}\label{lem-4-2}
Let $G$ be an $1^+$-deficient connected graph except graphs in $\mathcal{U}$ with $B$ as a block of order at least $3$. If $G$ is neither a $\theta$-graph nor an $\infty$-graph, then $B$ is a pendant cycle of $G$ and the major vertex of $B$ has degree $3$.
\end{lem}
\begin{proof}
Assume $\mathcal{B}$ is the set of blocks of $G$ in which each block $B$ has at least two major vertices or exactly one major vertex with degree at least $4$ of $G$.
To prove this theorem, it suffices to show that $\mathcal{B}=\varnothing$.
On the contrary, choose $B^*$ as the block of $G$ with the largest number of major vertices in $\mathcal{B}$. In the following, we proceed by induction on $\theta(G)$ to prove $G$ is $2^+$-deficient.
If $\theta(G)=1$, by Lemma \ref{thm-4}, then $G$ is $2^+$-deficient since $B^*\in \mathcal{B}$.
Suppose all graphs with a block $B\in \mathcal{B}$ and cyclomatic number no more than $t-1~(t\geq2)$ are $2^+$-deficient graphs. Let $G$ be a graph with $\theta(G)=t$.
{\flushleft  \bf Case 1.} If there is a cycle  $C'$ of $G$ with $C'\nsubseteq B^*$.
{\flushleft  \bf Subcase 1.} All vertices of $C'$ are major vertices.

In this situation, we can choose an edge on $C'$, say $e^*$. Clearly, $\theta(G-e^*)=\theta(G)-1=t-1$ and $p(G-e^*)=p(G)$.
Since $G$ is not an $\infty$-graph and $B^*$ is the block of $G$ with the largest number of major vertices in $\mathcal{B}$, we have
 $B^*\subseteq G-e^*$ and $G-e^*\notin\mathcal{U}$.
By Lemma \ref{lem-2-12} and the induction hypothesis, we have
\begin{align*}
m_B(G, \lambda)&\leq m_B(G-e^*, \lambda)+2\\
&\leq2\theta(G-e^*)+p(G-e^*)-2+2\\
&=2(\theta(G)-1)+p(G)\\
&=2\theta(G)+p(G)-2.
\end{align*}
{\flushleft \bf Subcase 2.} There exists a vertex with degree $2$ of $C'$.

In this situation, we select a vertex $v^*$ with degree $2$ that is adjacent to at least one major vertex of $C'$.
It is evident that $\theta(G-v^*)=\theta(G)-1=t-1$ and $p(G-v^*)\leq p(G)+1$.
Note that $B^*\subseteq G-v^*$ and $B^*\in \mathcal{B}$.
Since $G$ is not an $\infty$-graph and $B^*$ is the block of $G$ with the largest number of major vertices in $\mathcal{B}$, we can conclude that $G-v^*\notin\mathcal{U}$. By Lemma \ref{lem-2-11} and the induction hypothesis,
\begin{equation*}
\begin{array}{lll}
m_B(G, \lambda)&\leq& m_B(G-v^*, \lambda)+1\\
&\leq&2\theta(G-v^*)+p(G-v^*)-2+1\\
&\le&2(\theta(G)-1)+(p(G)+1)-1\\
&=&2\theta(G)+p(G)-2.
\end{array}
\end{equation*}

{\flushleft  \bf Case 2.} There is only one block $B^*$ with order at least $3$ of $G$.

Note that $\theta(G)\geq2$ and $G$ is not a $\theta$-graph. Then $\theta(B^*)\geq 2$ and there at least three major vertex of $G$.
If all vertices of $B^*$ are major vertices, as similar analysis as Subcase $1$, $G$ is $2^+$-deficient.
Otherwise, there exists a vertex with degree $2$ of $B^*$, say $v^*$,  adjacent to at least one major vertices of $B^*$.
Since $G$ is not a $\theta$-graph, we have $G-v^*$ is not isomorphic to $C_{n-1}$.

If $v^*$ is adjacent to two major vertices of $B^*$, then
$\theta(G-v^*)=\theta(G)-1=t-1$ and $p(G-v^*)=p(G)$.
By Lemma \ref{lem-2-11} and Theorem \ref{thm3},
\begin{align*}
m_B(G, \lambda)&\leq m_B(G-v^*, \lambda)+1\\
&\leq2\theta(G-v^*)+p(G-v^*)-1+1\\
&=2(\theta(G)-1)+p(G)\\
&=2\theta(G)+p(G)-2.
\end{align*}
If $v^*$ is adjacent to exactly one major vertices of $B^*$, then
$\theta(G-v^*)=\theta(G)-1=t-1$, $p(G-v^*)=p(G)+1$ and $B^*-v^*$ has a block containing at least two major vertices or exactly one major vertex with degree at least $4$ of $G$.
By Lemma \ref{lem-2-11} and the induction hypothesis, we have
\begin{align*}
m_B(G, \lambda)&\leq m_B(G-v^*, \lambda)+1\\
&\leq2\theta(G-v^*)+p(G-v^*)-2+1\\
&=2(\theta(G)-1)+p(G)\\
&=2\theta(G)+p(G)-2.
\end{align*}

Therefore, the result follows from the principle of induction.
\end{proof}


\begin{lem}\label{lem-4-4}
Let $G$ be either a $\theta$-graph or an $\infty$-graph, and let $y$ be the vertex of degree $2$ adjacent to a major vertex of $G$. If $G$ is $1^+$-deficient, then $$m_B(G-y, \lambda)=m_B(G, \lambda)-1=2.$$
Moreover, if $G$ is a $\theta$-graph, then $m_B(G-x, \lambda)=2$ for any $x\in \mathcal{X}(G)$.
\end{lem}
\begin{proof}
According to Theorem \ref{thm3}, we have $m_B(G-y, \lambda)\leq 2$.
On the other hand, by Lemma \ref{lem-2-11}, if $G$ is $1^+$-deficient, then
$m_B(G-y, \lambda)\ge m_B(G, \lambda)-1=2$.
\end{proof}

\begin{lem}\label{lem-4-5}
Let $G$ be a connected graph (not a cycle). If $G$ is $1^+$-deficient, then for any cycle-vertex $x$ of degree $2$ and adjacent to a major cycle-vertex of $G$, $$m_B(G, \lambda)=m_B(G-x, \lambda)+1.$$
Moreover, the graph $G-x$ also is $1^+$-deficient.
\end{lem}
\begin{proof}
According to Lemma \ref{lem-4-4}, we only need to consider the case when $G$ is neither a $\theta$-graph nor an $\infty$-graph. We will prove this by contradiction.
Using Lemma \ref{lem-4-2}, if $G$ is $1^+$-deficient, then every cycle in $G$ is a pendant cycle. Consequently, we have $p(G-x)=p(G)+1$ and $\theta(G-x)=\theta(G)-1$. Moreover, since $G$ is not a $\theta$-graph, $G-x$ cannot be a cycle. If $m_B(G, \lambda)\leq m_B(G-x, \lambda)$, then Theorem \ref{thm3} implies that
\begin{align*}
m_B(G, \lambda)&\leq m_B(G-x, \lambda)\\
&\leq 2\theta(G-x)+p(G-x)-1\\
&=2(\theta(G)-1)+p(G)+1-1\\
&=2\theta(G)+p(G)-2,
\end{align*}
which is a contradiction.
Therefore, $m_B(G, \lambda)=m_B(G-x, \lambda)+1$, and so
\begin{align*}
m_B(G-x, \lambda)&=m_B(G, \lambda)-1\\
&= 2\theta(G)+p(G)-1-1\\
&=2(\theta(G-x)+1)+p(G-x)-1-1-1\\
&=2\theta(G-x)+p(G-x)-1.
\end{align*}
This completes the proof.
\end{proof}

\begin{thm}\label{thm-5}
Let $G$ be a connected graph with $\theta(G)\geq 1$ and $G\notin \mathcal{U}$.
If $\lambda$ is an eigenvalue of $B(G)\in \mathcal{S}(G)$,
then $m_B(G, \lambda)= 2\theta(G)+p(G)-1$ if and only if  $G$  has one of the following forms:
\begin{enumerate}[(a)]
\vspace{-0.2cm}
   \item An $\infty$-graph or a $\theta$-graph, and for any vertex $y$ of degree 2 and adjacent to a major vertex of $G$, $m_B(G-y, \lambda)=m_B(G, \lambda)-1=2$. Moreover, if $G$ is a $\theta$-graph, then $m_B(G-x, \lambda)=2$ for any vertex $x\in \mathcal{X}(G)$.
\vspace{-0.2cm}
   \item There is only one major vertex in each cycle of $G$, and $G-\mathcal{M}(G)\cong \cup_{j=1}^{\theta(G)} C^*_j \cup_{i=1}^l P^i$, where any pair of major vertices of $\mathcal{M}(G)$ are not adjacent,  $C^*_j\in\mathcal{U}$ and $P^i$ is a path. Moreover, $\lambda\in \sigma(B(P^i))$ for any $1\leq i\leq l$  and  $m_B(C^*_j)=2$ for any $1\leq j \leq \theta(G)$.
\vspace{-0.2cm}
\end{enumerate}
\end{thm}
\begin{proof}

Sufficiency: For graphs of form (a), it is clear that $m_B(G,\lambda)=3$.

For graphs of form $(b)$, suppose $|\mathcal{M}(G)|=t$. Since $G-\mathcal{M}(G)\cong \cup_{j=1}^{\theta(G)} C^*_j \cup_{i=1}^l P^i$, where any two major vertices of $\mathcal{M}(G)$ are not adjacent and $C^*_j\in\mathcal{U}$, we have $p(G)=l-t+1$.  According to Lemma \ref{lem-2-11}, we have
\begin{align*}
m_B(G, \lambda)&\geq m_B(G-\mathcal{M}(G), \lambda)-t\\
&=2\theta(G)+l-t\\
&=2\theta(G)+p(G)-1.
\end{align*}
This together with Theorem \ref{thm3} gives that $m_B(G,\lambda)=2\theta(G)+p(G)-1$.

Necessity:
We have (a) if $G$ is an $\infty$-graph or a $\theta$-graph by Lemma \ref{lem-4-4}.

Assume that $G$ is neither an $\infty$-graph or $\theta$-graph.
We proceed by induction on $\theta(G)$ to show that $G$ has the form of $(b)$.
When $\theta(G)=1$, it has been proved by Lemma \ref{thm-4}. Suppose the result holds for all connected graphs with $1\leq \theta(G)\leq k~(k\geq1)$ and let $G$ be a graph with $\theta(G)=k+1$.
Note that all cycles of $G$ are pendant cycles by Lemma \ref{lem-4-2}.
For two pendant cycles  $C^i$ and $C^j$, suppose $x^i$ and $x^j$ are cycle-verteices  adjacent to the major cycle-vertex of $C^i$ and $C^j$, respectively.
It is easy to notice $G-x^i, G-x^j\notin \mathcal{U}$. According to Lemma \ref{lem-4-5}, we have both $G-x^i$ and $G-x^j$ are $1^+$-deficient. So,
by the induction hypothesis, $G-x^i$ and $G-x^j$ have the form of $(b)$, implying that $G$ has the form of $(b)$.
By the principle of induction, the proof is complete.
\end{proof}

To characterize all $1^+$-deficient connected graphs, we conclude the following result including Theorems \ref{thm-2} and \ref{thm-5}.

\begin{thm}\label{thm-conclusion-result}
Let $G$ be a connected graph. If $\lambda$ is an eigenvalue of $B(G)\in \mathcal{S}(G)$,
then $m_B(G, \lambda)= 2\theta(G)+p(G)-1$ if and only if  $G$  has one of the following forms:
\begin{enumerate}[(a)]
\vspace{-0.2cm}
\item A path, or a tree $T$ with $p(T)\ge3$ and each path $P_i$ of $T-\mathcal{X}(T)$ with $\lambda\in \sigma(A(P_i))$ and any two major vertices of $\mathcal{X}(T)$ are not adjacent.
\vspace{-0.2cm}
\item A cycle with $m_B(G,\lambda)=1$, or a graph in $\mathcal{U}$ except a cycle with $m_B(G,\lambda)=2$.
\vspace{-0.2cm}
   \item An $\infty$-graph or a $\theta$-graph, and for any vertex $y$ of degree 2 and adjacent to a major vertex of $G$, $m_B(G-y, \lambda)=m_B(G, \lambda)-1=2$. Moreover, if $G$ is a $\theta$-graph, then $m_B(G-x, \lambda)=2$ for any vertex $x\in \mathcal{X}(G)$.
\vspace{-0.2cm}
   \item A graph with $\mathcal{M}\not=\varnothing$ and $\theta(G)\ge 1$ such that there is only one major vertex in each cycle of $G$, and $G-\mathcal{M}(G)\cong \cup_{j=1}^{\theta(G)} C^*_j \cup_{i=1}^l P^i$, where any pair of major vertices of $\mathcal{M}(G)$ are not adjacent, $C^*_j\in\mathcal{U}$ and $P^i$ is a path.
   Moreover, $\lambda\in \sigma(B(P^i))$ for any $1\leq i\leq l$  and  $m_B(C^*_j)=2$ for any $1\leq j \leq \theta(G)$.
\end{enumerate}
\end{thm}

Evidently, we may not provide a more detailed description for $(b)$ of Theorem \ref{thm-conclusion-result} in view of a constructive counterexample,
since $m_{A}(C_4,0)=2$ but $m_{B}(C_4,\lambda)=1$ for any eigenvalue of $B(C_4)\in \mathcal{S}(C_4)$, where $B$ is the matrix $A$ with $a_{12}$ and $a_{21}$ replaced by $2$.
This observation highlights that the multiplicity of an eigenvalue of $B(G)$ for any graph in $\mathcal{U}$ is significantly influenced by the matrix $B(G)$.
In spite of this observation, we can investigate the constitutive properties of $G$ when $B$ is considered as the adjacency matrix.

Let $A(G)$ be the adjacency matrix of $G$. Next, we will characterize all graphs in $\mathcal{U}$ with $m_A(C^*, \lambda)=2$.
We may notice that $m_A(C_n,\lambda)=2$ if and only if $\lambda=2\cos\frac{2j\pi}{n}$ for $j=1,\ldots,\lfloor\frac{n}{2}\rfloor$ (see \cite{Cv2} for example).

\begin{lem}\label{lem-4-1}
Let $C^*=C_mwP^*$ be a unicyclic graph obtained from an isolated vertex $w$ and the union of a cycle $C_m$ and a path $P^*$  (possibly, $P^*=\varnothing$), by connecting $w$ and a vertex of $C_m$, 
and connecting $w$ and a pendent vertex of $P^*$ (see $Fig.\,\ref{fig-3}$).
Then $m_A(C^*, \lambda)= 2$ if and only if $\lambda\in \sigma(A(P^*))$  and  $m_A(C_m, \lambda)=2$.
\end{lem}
\begin{proof}
Necessity: Since $m_A(C^*, \lambda)=2$, we can construct an eigenvector $Y\neq \textbf{0}$ of $A(C^*)$ corresponding to $\lambda$ with the component $Y(w)=0$. Let $Y_1$ and $Y_2$ be the sub-vector of $Y$ corresponding to $C_m$ and $P^*$, respectively.
If $Y_2\not=\textbf{0}$, then $Y_2$ is a solution of $A(P^*)X=\lambda X$, which implies that $\lambda\in \sigma(A(P^*))$.
This gives $\lambda\notin\sigma(A(P^*-u))$ from Lemma \ref{lem-2-2}, where $u$ is the vertex of $P^*$ adjacent to $w$.
Then, by Lemma \ref{lem-2-1}, we have $2=m_B(C^*, \lambda)=m_A(C^*-P^*-w, \lambda)=m_A(C_m, \lambda)$.

If $Y_2=\textbf{0}$, then $Y_1\neq\textbf{0}$ for $Y\neq \textbf{0}$, which implies that $\lambda\in \sigma(A(C_m))$.
Moreover, from the eigen-equation $A(C^*)Y=\lambda Y$, we have $Y(v_1)=0$, which implies $\lambda\in \sigma(A(C_m-v_1))$. Let $v_2$ be the cycle-vertex adjacent to $v_1$.
Clearly, $\lambda\in \sigma(A(C_m-v_2))$, that is, $\lambda\in \sigma(A(P_{m-1}))$.
Notice that $C^*-v_2=P_{m-1}wP^*$.
If $\lambda\notin \sigma(A(P^*))$, by Lemmas \ref{lem-2-11}, \ref{lem-2-1} and \ref{lem-2-2}, we obtain $$m_A(C^*, \lambda)\leq m_A(C^*-v_2, \lambda)+1=m_A(P^*, \lambda)+1=1,$$
a contradiction.
Therefore, we have $\lambda\in\sigma(A(P^*))$. Then  $\lambda\notin \sigma(A(P^*-u))$ by Lemma \ref{lem-2-2}, which follows that $m_A(C_m,\lambda)=m_A(C^*,\lambda)=2$.

Sufficiency: If $\lambda\in \sigma(A(P^*))$  and  $m_A(C_m, \lambda)=2$, then by Lemma \ref{lem-2-1}, we have $m_A(C^*, \lambda)=m_A(C^*-P^*-w, \lambda)=m_A(C_m, \lambda)=2$.
\end{proof}

\begin{remark}
Lemma \ref{lem-4-1} not holds for all  matrix $B\in \mathcal{S}(G)$. Let $H_1$ and $H_2$ be the graphs shown as $Fig.\,\ref{fig-3}$, and
$\begin{matrix}
 B(H_1)=\left(\begin{array}{cccc}
 -10 & -10 & -10 & 8\\
-10 &-3 & -5&0 \\
  -10 & -5 & -3&0\\
  8&0&0&10\\
 \end{array}\right),
 \end{matrix}$
 $\begin{matrix}
 B(H_2)=\left(\begin{array}{cccccc}
 0 & 1 & 1 & 8 & 0 & 0\\
 1 & 0 & 9 & 0 & 0 & 0\\
 1 & 9 & 0 & 0 & 0 & 0\\
 8 & 0 & 0 & 0 & 4 & 0\\
 0 & 0 & 0 & 4 & 0 & 1\\
 0 & 0 & 0 & 0 & 1 & 0\\
 \end{array}\right).
 \end{matrix}$
By straightforward calculation, it can be shown that  $m_B(H_1, 2)=2$ and $m_B(H_2, -9)=2$. However,  we observe that
 $m_B(H_1[\{v_1,v_2,v_3\}],2)=1$, $m_B(H_2[\{v_1,v_2,v_3\}],-9)=1$, and $m_B(H_2[\{v_5,v_6\}],-9)=0$.
\end{remark}

\vskip0.3in
\begin{figure}[htbp]
\centering
\begin{tikzpicture}[x=1.00mm, y=1.00mm, inner xsep=0pt, inner ysep=0pt, outer xsep=0pt, outer ysep=0pt,scale=0.7]
\path[line width=0mm] (92.83,32.29) rectangle +(151.37,69.57);
\definecolor{L}{rgb}{0,0,0}
\path[line width=0.30mm, draw=L] (106.82,87.55) circle (11.36mm);
\definecolor{F}{rgb}{0,0,0}
\path[line width=0.30mm, draw=L, fill=F] (106.59,75.80) circle (1.00mm);
\path[line width=0.30mm, draw=L, fill=F] (106.82,67.10) circle (1.00mm);
\path[fill=F] (106.59,59.58) circle (1.00mm);
\path[fill=F] (106.35,44.30) circle (1.00mm);
\path[fill=F] (101.42,77.44) circle (1.00mm);
\path[] (106.59,76.03) -- (106.82,66.63);
\path[line width=0.30mm, draw=L] (106.82,76.27) -- (106.82,67.57);
\path[line width=0.30mm, draw=L, fill=F] (156.04,95.86) circle (1.00mm);
\path[line width=0.30mm, draw=L, fill=F] (176.01,95.86) circle (1.00mm);
\path[line width=0.30mm, draw=L, fill=F] (166.61,76.82) circle (1.00mm);
\path[line width=0.30mm, draw=L, fill=F] (166.61,56.85) circle (1.00mm);
\path[line width=0.30mm, draw=L] (155.80,95.63) -- (176.25,95.63);
\path[line width=0.30mm, draw=L] (156.04,95.39) -- (166.61,75.65);
\path[line width=0.30mm, draw=L] (176.25,95.39) -- (166.85,76.35);
\path[line width=0.30mm, draw=L] (166.85,76.82) -- (166.85,56.61);
\path[line width=0.30mm, draw=L, dash pattern=on 0.30mm off 0.50mm] (106.59,59.58) -- (106.59,44.07);
\path[line width=0.30mm, draw=L] (106.59,67.57) -- (106.59,58.64);
\draw(168.73,74.79) node[anchor=base west]{\fontsize{8.54}{10.24}\selectfont $v_1(-10)$};
\draw(142.77,94.22) node[anchor=base west]{\fontsize{8.54}{10.24}\selectfont $(-3)v_2$};
\draw(177.55,94.30) node[anchor=base west]{\fontsize{8.54}{10.24}\selectfont $v_3(-3)$};
\draw(168.73,55.91) node[anchor=base west]{\fontsize{8.54}{10.24}\selectfont $v_4(10)$};
\draw(153.22,83.87) node[anchor=base west]{\fontsize{8.54}{10.24}\selectfont $-10$};
\draw(172.49,83.64) node[anchor=base west]{\fontsize{8.54}{10.24}\selectfont $-10$};
\draw(163.55,97.51) node[anchor=base west]{\fontsize{8.54}{10.24}\selectfont $-5$};
\draw(108.94,65.46) node[anchor=base west]{\fontsize{8.54}{10.24}\selectfont $w$};
\draw(101.65,84.73) node[anchor=base west]{\fontsize{11.38}{13.66}\selectfont $C_m$};
\draw(108.76,73.45) node[anchor=base west]{\fontsize{8.54}{10.24}\selectfont $v_1$};
\draw(97.83,73.68) node[anchor=base west]{\fontsize{8.54}{10.24}\selectfont $v_2$};
\draw(108.70,51.35) node[anchor=base west]{\fontsize{8.54}{10.24}\selectfont $P^*$};
\draw(102.83,35.37) node[anchor=base west]{\fontsize{14.23}{17.07}\selectfont $C^*$};
\draw(163.44,35.37) node[anchor=base west]{\fontsize{14.23}{17.07}\selectfont $H_1$};
\draw(168.26,64.10) node[anchor=base west]{\fontsize{8.54}{10.24}\selectfont $8$};
\path[line width=0.30mm, draw=L, fill=F] (212.05,95.68) circle (1.00mm);
\path[line width=0.30mm, draw=L, fill=F] (229.11,95.53) circle (1.00mm);
\path[line width=0.30mm, draw=L, fill=F] (220.43,76.97) circle (1.00mm);
\path[line width=0.30mm, draw=L, fill=F] (220.58,66.49) circle (1.00mm);
\path[line width=0.30mm, draw=L, fill=F] (220.58,44.19) circle (1.00mm);
\path[line width=0.30mm, draw=L, fill=F] (220.73,55.12) circle (1.00mm);
\path[line width=0.30mm, draw=L] (211.90,95.83) -- (229.41,95.68);
\path[line width=0.30mm, draw=L] (211.90,95.98) -- (220.73,76.97);
\path[line width=0.30mm, draw=L] (229.11,95.98) -- (220.58,76.52);
\path[line width=0.30mm, draw=L] (220.43,77.12) -- (220.58,66.34);
\path[line width=0.30mm, draw=L] (220.58,66.94) -- (220.58,54.52);
\path[line width=0.30mm, draw=L] (220.73,55.27) -- (220.73,43.60);
\draw(222.52,74.43) node[anchor=base west]{\fontsize{8.54}{10.24}\selectfont $v_1(0)$};
\draw(201.62,94.33) node[anchor=base west]{\fontsize{8.54}{10.24}\selectfont $(0)v_2$};
\draw(231.20,94.33) node[anchor=base west]{\fontsize{8.54}{10.24}\selectfont $v_3(0)$};
\draw(219.44,97.03) node[anchor=base west]{\fontsize{8.54}{10.24}\selectfont $9$};
\draw(212.20,84.90) node[anchor=base west]{\fontsize{8.54}{10.24}\selectfont $1$};
\draw(225.82,84.75) node[anchor=base west]{\fontsize{8.54}{10.24}\selectfont $1$};
\draw(222.52,65.60) node[anchor=base west]{\fontsize{8.54}{10.24}\selectfont $v_4(0)$};
\draw(222.52,54.22) node[anchor=base west]{\fontsize{8.54}{10.24}\selectfont $v_5(0)$};
\draw(222.52,43.30) node[anchor=base west]{\fontsize{8.54}{10.24}\selectfont $v_6(0)$};
\draw(221.63,70.09) node[anchor=base west]{\fontsize{8.54}{10.24}\selectfont $8$};
\draw(221.63,59.31) node[anchor=base west]{\fontsize{8.54}{10.24}\selectfont $4$};
\draw(221.63,49.13) node[anchor=base west]{\fontsize{8.54}{10.24}\selectfont $1$};
\draw(109.23,58.56) node[anchor=base west]{\fontsize{8.54}{10.24}\selectfont $u$};
\draw(218.03,35.37) node[anchor=base west]{\fontsize{14.23}{17.07}\selectfont $H_2$};
\end{tikzpicture}%
\caption{The unicyclic graph $C^*$ and two exemplary graphs $H_1$ and $H_2$}\label{fig-3}
\end{figure}
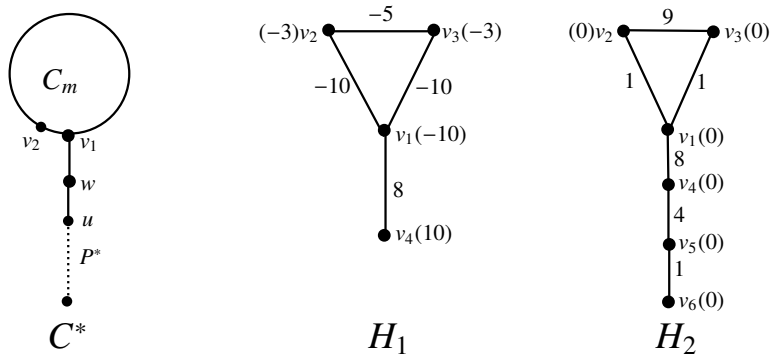
\vskip0.3in

According to Theorem \ref{thm-5} and Lemma \ref{lem-4-1}, we may deduce many known results obtained in \cite{Chang1}, \cite{Wang4} (\cite{Chang}) and \cite{Zhou}.
\begin{cor}[\cite{Zhou}]
Let $G$ be a connected graph (not a cycle) with $\theta(G) \geq 1$.
Then $m_A(G,-1)=$ $2 \theta(G)+p(G)-1$ if and only if it is obtained from a tree $T$ with $m_A(T,-1)=p(T)-1$ by attaching $\theta(G)$ cycles of order a multiple of 3 to $\theta(G)$ quasi-pendent vertices of $T$ and delete related $\theta(G)$ pendent vertices,
called construction condition, where $p(T) \geq \theta(G)$.
\end{cor}
\begin{proof}
Note that $m_A(P_n, -1)=1$ if and only if $n=3k-1~(k\geq1)$ and $m_A(C_n, -1)=2$ if and only if $n=3k~(k\geq1)$. By utilizing Corollary \ref{cor-2} and Lemma \ref{lem-4-1}, we see that the condition $(d)$ of Theorem \ref{thm-conclusion-result} is equivalent to the construction condition when $\mathcal{M}\not=\varnothing$.
And when $G$ is a graph in $\mathcal{U}$, the result follows from Theorem \ref{thm-conclusion-result} $(b)$ and Lemma \ref{lem-4-1}.
In the case that $G$ is a $\theta$-graph, let $x, y$ be the only two vertices in $\mathcal{X}(G)$, and let $w$ be a vertex adjacent to $x$ on the path $P_{l+1}$.
According to Theorem \ref{thm-conclusion-result} $(c)$, if $x\in \mathcal{X}(G)$, then $m_A(G-x, -1)=2$.
This implies that $G-x$ is not a path and any path of $G-x-y$ has $-1$ as an eigenvalue by Theorem \ref{thm-2}, implying $l\equiv0\,(\bmod \,3)$.
On the other hand, since $m_A(G-w, 2)=2$ by Theorem \ref{thm-conclusion-result} $(c)$, we have $-1\in \sigma(A(P_{l-3}))$ by Lemma \ref{lem-4-1}, implying $l\equiv2\,(\bmod\,3)$, a contradiction.
As a result, we conclude that $G$ cannot be a $\theta$-graph.
If $G$ is an $\infty$-graph, then $p,q,l\equiv 0\,(\bmod\,3)$ according to Theorem \ref{thm-conclusion-result} $(c)$ and Lemma \ref{lem-4-1}, which coincides with the construction conditions.
\end{proof}

\begin{cor}[\cite{Chang1,Chang,Wang4}]
 Let $G$ be a connected graph with $\theta(G) \geq 1$. Then $\eta(G)=2\theta(G)+p(G)-1$ if and only if $G$ has one of the following forms.
 \begin{enumerate}[(a)]
\vspace{-0.2cm}
   \item A graph obtained from a tree $T$ with nullity $p(T)-1$ by attaching $c(G)$ cycles of orders multiple of $4$ at $c(G)$ leaves of $T$, where $p(T) \geq c(G)$;
\vspace{-0.2cm}
   \item An $\infty$-graph $\infty(p, q, l)$, where $p \equiv q \equiv 0\,(\bmod\, 4)$ and $l \equiv 1\,(\bmod\, 2)$;
\vspace{-0.2cm}
   \item A $\theta$-graph $\theta(p, q, l)$, where $p \equiv q \equiv l \equiv 0\,(\bmod\, 4)$ or $p \equiv q \equiv l \equiv 2\,(\bmod\, 4)$.
\end{enumerate}
\end{cor}
\begin{proof}
Note that $\eta(P_n)=1$ if and only if $n$ is odd, and $\eta(C_n)=2$ if and only if $n$ is a multiple of $4$.
By Corollary \ref{cor-1} and Lemma \ref{lem-4-1}, condition $(d)$ of Theorem \ref{thm-conclusion-result} is equivalent to $(a)$.
If $G\cong \infty(p,q,l)$, by Theorem \ref{thm-conclusion-result} $(c)$, since for any vertex $y$ adjacent to the major vertex of $G$, $m_A(G-y, \lambda)=2$, and $m_A(G-x, -1)=2$ for $x\in \mathcal{X}(G)$, we have $p\equiv q\equiv l\equiv 0\,(\bmod\, 2)$, and $p+q\equiv p+l\equiv q+l\equiv 0\,(\bmod\, 4)$ by Lemma \ref{lem-4-1}.
Hence $p\equiv q\equiv l\equiv 0\,(\bmod\, 4)$ or $p\equiv q\equiv l\equiv 2\,(\bmod\, 4)$.
If $G\cong \theta(p, q, l)$, according to Theorem \ref{thm-conclusion-result} $(c)$ and Lemma \ref{lem-4-1}, then $p\equiv q\equiv 0\,(\bmod\, 4)$ and $l\equiv1\,(\bmod\, 2)$.
\end{proof}

\vskip 0.3in
\noindent{\bf\large Declaration of competing interest}
\vskip0.1in
No conflict of interest to this work.


\begin{thebibliography}{}

\bibitem{Atkins} P. W. Atkins, J. de Paula, Physical Chemistry, eighth ed., Oxford University Press, 2006.

\bibitem{CV} D. Cvetkovi\'{c}, I. Gutman, The algebraic multiplicity of the number zero in the spectrum of a bipartite graph, Mat. Vesnik  9 (24) (1972) 141--150.

\bibitem{Chang1}S. Chang, A. Chang, Y. Zheng, The leaf-free graphs with nullity $2c(G)-1$, \textit{Discrete Appl. Math.} 277 (2020) 44--54.

\bibitem{Chang} S. Chang, B. Tam, J. Li, Y. Zheng,  Graphs $G$ with nullity $2c(G)+p(G)-1$, \textit{Discrete Appl. Math.} 311 (2022) 38--58.


\bibitem{Chen}Q. Chen, J. Guo, The nullity of a graph with fractional matching number, \textit{Discrete Math.} 345 (2022), no. 8, 112919.


\bibitem{Cv2} D. Cvetkovi\'{c}, P. Rowlinson, S. Simi\'{c}, \textit{An introduction to the Theory of Graph Spectra}, Cambridge University Press, Cambridge, 2010.

\bibitem{Johnson3} A. L. Duarte, C. R. Johnson,
On the minimum number of distinct eigenvalues for a symmetric matrix whose graph is a given tree, \textit{Math. Inequal. Appl.} 5 (2002), no. 2, 175--180.

\bibitem{C.M.}C. M. da Fonseca, A note on the multiplicities of the eigenvalues of a graph, \textit{Linear Multilinear Algebra} 53 (2005) 303--307.

\bibitem{C.M.1}C. M. da Fonseca, A lower bound for the number of distinct eigenvalues of some real symmetric matrices, \textit{Electron. J. Linear Algebra} 21 (2010) 3--11.

\bibitem{Belardo} B. Francesco, B. Maurizio, A. Ciampella, On the multiplicity of $\alpha$ as an $A_\alpha(\Gamma)$-eigenvalue of signed graphs with pendant vertices, \textit{Discrete Math.} 342 (2019), no. 8, 2223--2233.


\bibitem{Genin} J. Genin, J. S. Maybee, Mechanical vibration trees, \textit{J. Math. Anal. Appl.} 45 (1974) 746--763.

\bibitem{Gong}S. Gong, G. Xu, On the nullity of a graph with cut-points, \textit{Linear Algebra Appl.} 436 (2012) no. 1, 135--42.

 \bibitem{Guo}J. Guo, W. Yan, Y. N. Yeh, On the nullity and the matching number of unicyclic graphs, \textit{Linear Algebra Appl.}431 (2009), no. 8, 1293--1301.

\bibitem{Johnson1}C. R. Johnson, A. L. Duarte,
The maximum multiplicity of an eigenvalue in a matrix whose graph is a tree,
\textit{Linear Multilinear Algebra} 46 (1999) 139--144.

\bibitem{Johnson2}C. R. Johnson, A. L. Duarte, L. Ant\'{o}nio, C. M. Saiago, B. D. Sutton, A. J. Witt, On the relative position of multiple eigenvalues in the spectrum of an Hermitian matrix with a given graph, \textit{Linear Algebra Appl.} 363 (2003) 147--159.


\bibitem{Li}S. Li, W. Wei, The multiplicity of an $A_\alpha$-eigenvalue: a unified approach for mixed graphs and complex unit gain graphs, \textit{Discrete Math.} 343 (2020), no. 8, 111916.

\bibitem{Ma} X. Ma, D. Wong, F. Tian, Nullity of a graph in terms of the dimension of cycle space and the number of pendent vertices, \textit{Discrete Appl. Math.} 215 (2016) 171--176.

\bibitem{Parter} S. Parter, On the eigenvalues and eigenvectors of a class of matrices, \textit{J. Soc. Indust. Appl. Math.} 8 (1960) 376--388.

\bibitem{Wang4}L. Wang, X. Fang, X. Geng, Graphs with nullity $2c(G)+p(G)-1$, \textit{Discrete Math.} 345 (2022), no. 5,  112786.

\bibitem{Wang1}L. Wang, L. Wei, Y. Jin, The multiplicity of an arbitrary eigenvalue of a graph in terms of cyclomatic number and number of pendant vertices, \textit{Linear Algebra Appl.} 584 (2020) 257--266.


\bibitem{Wang5}L. Wang, D. Wong, Bounds for the matching number, the edge chromatic number and the independence number of a graph in terms of rank, \textit{ Discrete Appl. Math. } 166 (2014) 276--281.

\bibitem{Wang3}X. Wang, D. Wong, L. Wei, F. Tian, On the multiplicity of $-1$ as an eigenvalue of a tree with given number of pendant vertices, \textit{Linear Multilinear Algebra} 70 (2022), no. 17, 3345--3353.


\bibitem{Xu}F. Xu, D. Wong, F. Tian, On the multiplicity of $\alpha$ as an eigenvalue of the $A_\alpha$ matrix of a graph in terms of the number of pendant vertices, \textit{Linear Algebra Appl.} 594 (2020) 193--204.

\bibitem{Zhou}Q. Zhou, D. Wong, Graphs with eigenvalue $-1$ of multiplicity $2c(G)+p(G)-1$, \textit{Linear Algebra Appl.} 645 (2022) 137--152.
\end{thebibliography}
\end{document}